\documentclass[12pt]{amsart}
\usepackage{amssymb,amsmath,amsthm,mathrsfs,MnSymbol,enumitem}
\usepackage{graphicx}
\usepackage[all,cmtip]{xy}
\usepackage[table]{xcolor}
\usepackage{multirow,scalerel,arydshln,caption,subcaption}   
\usepackage{hyperref}
\usepackage[msc-links]{amsrefs}
\hypersetup{
  colorlinks   = true,	
  urlcolor     = blue,
  linkcolor    = blue,
  citecolor   = red
}
\calclayout
\allowdisplaybreaks
\newtheorem{theorem}{Theorem}
\newtheorem{lemma}[theorem]{Lemma}
\newtheorem{remark}[theorem]{Remark}
\newtheorem{corollary}[theorem]{Corollary}

\newtheorem{proposition}[theorem]{Proposition}
\numberwithin{theorem}{section} \numberwithin{equation}{section}
\newcommand{\beq}{\begin{small} \begin{equation}}
\newcommand{\eeq}{\end{equation} \end{small}}
\newcommand{\beqn}{\begin{small} \begin{equation*}}
\newcommand{\eeqn}{\end{equation*} \end{small}}
\DeclareMathAlphabet{\mathpzc}{OT1}{pzc}{m}{it}
\newcommand\scalemath[2]{\scalebox{#1}{\mbox{\ensuremath{\displaystyle #2}}}}
\newcommand*{\MyScaleBig}{1.1}
\newcommand*{\MyScaleMed}{0.9}

\newcommand*{\MyScaleSmall}{0.62}
\begin{document}
\title[Elliptic fibrations on the generalized Inose quartic]{Jacobian elliptic fibrations on the generalized Inose quartic of Picard rank sixteen}
\author{Adrian Clingher}
\address{Department of Mathematics and Statistics, University of Missouri -- St. Louis, St. Louis, MO 63121}
\email{clinghera@umsl.edu}
\thanks{A.C. acknowledges support from a UMSL Mid-Career Research Grant.}
\author{Thomas Hill}
\address{Department of Mathematics \& Statistics, Utah State University, Logan, UT 84322}
\email{thomas.hill@usu.edu}
\thanks{T.H. acknowledges the support from the Office of Graduate Studies at Utah State University.}
\author{Andreas Malmendier}
\address{Department of Mathematics \& Statistics, Utah State University, Logan, UT 84322}
\email{andreas.malmendier@usu.edu}
\thanks{A.M. acknowledges support from the Simons Foundation through grant no.~202367.}
\begin{abstract}
We consider the family of complex algebraic K3 surfaces $\mathcal{X}$ with Picard lattice containing the unimodular lattice $H \oplus E_7(-1) \oplus E_7(-1)$.  The surface $\mathcal{X}$ admits a birational model isomorphic to a quartic hypersurface that generalizes the Inose quartic. We prove that a general member of this family admits exactly four inequivalent Jacobian elliptic fibrations and construct explicit pencils for them.
\end{abstract}
\maketitle
\section{Introduction and Summary of results}
Let $\mathcal{X}$ be a smooth complex algebraic K3 surface. Denote by $\operatorname{NS}(\mathcal{X})$ the N\'eron-Severi lattice of $\mathcal{X}$. This is known to be an even lattice of signature $(1,p_\mathcal{X}-1)$, where $p_\mathcal{X}$ being the Picard rank of $\mathcal{X}$, with $1 \leq p_\mathcal{X} \leq 20$.  A {\it lattice polarization} \cites{MR0357410,MR0429917,MR544937,MR525944,MR728992} on $\mathcal{X}$ is, by definition, a primitive lattice embedding $i \colon L \hookrightarrow \operatorname{NS}(\mathcal{X})$, with $i(L)$ containing a pseudo-ample class. Here, $L$ is a choice of even indefinite lattice of signature $(1,\rho_L-1)$, with $ 1 \leq \rho_L \leq 20$. Two $L$-polarized K3 surfaces $(\mathcal{X},i)$ and $(\mathcal{X}',i')$ are said to be isomorphic\footnote{Our definition of isomorphic lattice polarizations coincides with the one used by Vinberg \cites{MR2682724, MR2718942, MR3235787}. It is slightly more general than the one used in \cite{MR1420220}*{Sec.~1}.},  if there exists an analytic isomorphism $\alpha \colon \mathcal{X} \rightarrow \mathcal{X}'$ and a lattice isometry  $ \beta \in O(L)$, such that $ \alpha^* \circ i' = i \circ \beta $, where $\alpha^*$ is the induced morphism at cohomology level. In general, $L$-polarized K3 surfaces are classified, up to isomorphism, by a coarse moduli space $\mathscr{M}_{L}$, which is known  \cite{MR1420220} to be a quasi-projective variety of dimension $20-\rho_L$.   A \emph{general} $L$-polarized K3 surface $(\mathcal{X},i)$ satisfies $i(L)=\operatorname{NS}(\mathcal{X})$.
\par The present paper focuses on a special class of such objects -- K3 surfaces polarized by the rank sixteen lattice:
\beq 
 L \ = \ H \oplus E_7(-1) \oplus E_7(-1) \,.
\eeq 
Here $H$ stands for the unimodular hyperbolic lattice of rank two, and $E_7(-1)$ is the negative definite even lattice associated with the analogous root system.  This notation will be used throughout this article. The following is known about a general $L$-polarized K3 surface $\mathcal{X}$: Nikulin~\cite{MR633160} and Kondo~\cite{MR1139659} proved that i) the automorphism group $\mathrm{Aut}(\mathcal{X})$ is finite, or more precisely, Klein’s group of order 4, and ii) the number of $(-2)$-curves on $\mathcal{X}$  is 19 and their configuration forms a certain dual graph  that we will recall in~(\ref{exdiagg77}).
\par Our interest in this class of K3 surfaces is multi-fold. First, as observed in earlier work \cite{MR4015343} by the authors, K3 surfaces of this type  are  explicitly constructible. In fact, they fit into a six-parameter family of quartic {\it normal forms}: 
\begin{theorem}[\cite{MR4015343}]
\label{thm1_intro}
Let $(\alpha, \beta, \gamma, \delta , \varepsilon, \zeta) \in \mathbb{C}^6 $. Consider the projective quartic surface in $\mathbb{P}^3(\mathbf{X}, \mathbf{Y}, \mathbf{Z}, \mathbf{W})$ defined by the homogeneous equation:   
\beq
\label{quartic1}
\mathbf{Y}^2 \mathbf{Z} \mathbf{W}-4 \mathbf{X}^3 \mathbf{Z}+3 \alpha \mathbf{X} \mathbf{Z} \mathbf{W}^2+\beta \mathbf{Z} \mathbf{W}^3+\gamma \mathbf{X} \mathbf{Z}^2 \mathbf{W}- \frac{1}{2} \left (\delta \mathbf{Z}^2 \mathbf{W}^2+ \zeta \mathbf{W}^4 \right )+ \varepsilon \mathbf{X} \mathbf{W}^3 = 0 \,.
\eeq
Assume that $(\gamma, \delta, \varepsilon, \zeta) \neq 0$. Then, the surface  $\mathcal{X}(\alpha, \beta, \gamma, \delta , \varepsilon, \zeta)$ obtained as the minimal resolution of $(\ref{quartic1})$ is a K3 surface endowed with a canonical  $L$-polarization. Conversely, a general $L$-polarized K3 surface $\mathcal{X}$ has a birational projective model~(\ref{quartic1}).
\end{theorem}
\noindent Equation~(\ref{quartic1}) is a generalization of the \emph{Inose quartic}, a 2-parameter family first introduced by Inose in \cite{MR578868} that provides a  birational model for K3 surfaces with Picard lattice $H \oplus E_8(-1) \oplus E_8(-1)$.
\par As we will show, all $L$-polarized K3 surfaces, up to isomorphism, are in fact realized in this way. Moreover, one can tell precisely when two members of the above family are isomorphic.  Let $\mathpzc{G}$ be the subgroup of $\operatorname{Aut}(\mathbb{C}^6)$ generated by the set of transformations given below:
\beq
 (\alpha,\beta, \gamma, \delta, \varepsilon, \zeta) \ \longrightarrow \
(t^2 \alpha,  \ t^3 \beta,  \ t^5 \gamma,  \ t^6 \delta,  \ t^{-1} \varepsilon,  \ \zeta  ), \ {\rm with} \ t \in \mathbb{C}^* $$
$$ (\alpha,\beta, \gamma, \delta, \varepsilon, \zeta) \ \longrightarrow \  
(\alpha,  \beta,  \varepsilon,  \zeta, \gamma,  \delta  ) \,.
\eeq
It follows then that two K3 surfaces in the above family are isomorphic if and only if their six-parameter coefficient sets belong to the same orbit of $\mathbb{C}^6$ under $\mathpzc{G}$. This fact leads one  to define the following set of invariants associated to the K3 surfaces in the family:
\beq
\label{modinv}
J_2 = \alpha, \ \ \ J_3 = \beta, \ \ \ J_4 =  \gamma \cdot \varepsilon, \ \ \ J_5 = \gamma \cdot \zeta + \delta \cdot \varepsilon, \ \ \  J_6 = \delta \cdot \zeta
\eeq
The results above fit very well with the Hodge theory (periods) classification of these K3 objects. Hodge-theoretically,  lattice polarized K3 surfaces are well understood, by classical work of Dolgachev \cite{MR1420220} or $L$-polarized K3 surfaces, appropriate Torelli type arguments give a Hodge-theoretic coarse moduli space given by the modular four-fold $\Gamma_T^+ \backslash \mathbf{H}_2$, where the period domain $\mathbf{H}_2$ is a four-dimensional open domain of type $I_{2,2} \cong IV_4$ and $ \Gamma_T^+ $ is a discrete arithmetic group acting on $\mathbf{H}_2$; see work by Matsumoto~\cite{MR1204828}. In \cite{MR4015343}, the five invariants of $(\ref{modinv})$ were computed in terms of theta functions on $\mathbf{H}_2$.  Invariants of this type were independently obtained by Vinberg~\cite{MR2682724}.  These results then allows one to prove:
\begin{theorem}[\cite{MR4015343}]
 The four-dimensional open analytic space
\beq
\label{modulispace}
\mathscr{M}_L \ = \ \Big \{ \ 
\left [ \ J_2 , \ J_3 , \  J_4, \ J_5 , \ J_6  \ \right ]   \in  \mathbb{W}\mathbb{P}(2,3,4,5,6) \ \vert  \ 
( J_3 , \;  J_4, \; J_5 ) \neq (0,0,0) \ 
\Big \} \ 
\eeq
forms a coarse moduli space for $L$-polarized K3 surfaces. 
\end{theorem}
\par In this article, we will determine the Jacobian elliptic fibrations on a general $L$-polarized K3 surface $\mathcal{X}$ and show that their lattice theoretic multiplicities equal one. We then prove that there are exactly four inequivalent Jacobian elliptic fibrations on $\mathcal{X}$.  The uniqueness of one fibration, called the alternate fibration, implies that a general $L$-polarized K3 surface $\mathcal{X}$ has a birational projective model~(\ref{quartic1}), which is unique up to the action of the automorphism group. In Theorem~\ref{thm1}, we show that every Jacobian elliptic fibration on $\mathcal{X}$ is attained on the associated quartic projective surface in Equation~(\ref{quartic1}) as a pencil.  Equations for three types of elliptic fibrations are shown to be induced by pencils of lines. Equations of the fourth type of elliptic fibration is given by a pencil of quadric surfaces.
\par The article is structured as follows: First we note that Jacobian elliptic fibrations on $\mathcal{X}$ are related to primitive lattice embeddings $H \hookrightarrow L$. In Section~\ref{sec:lattice}, a lattice theoretic analysis of this problem reveals  that there are exactly four such (non-isomorphic) primitive lattice embeddings. An $L$-polarized K3 surface carries therefore, up to automorphisms, four special (non-isomorphic) Jacobian elliptic fibrations, the only ones existing in the general case. In the remaining part of the article we then explicitly identify the four fibrations and their corresponding pencils in the context of the normal forms $(\ref{quartic1})$.
\section{Lattice theoretic considerations for the K3 surfaces}
\label{sec:lattice}
Let $\mathcal{X}$ be a general $L$-polarized K3 surface with $L= H \oplus E_7(-1) \oplus E_7(-1)$. We start with a brief lattice-theoretic investigation regarding the possible Jacobian elliptic fibration structures appearing on the surface $\mathcal{X}$. Recall that a \emph{Jacobian elliptic fibration} on $\mathcal{X}$ is a pair $(\pi,\sigma)$ consisting of a proper map of analytic spaces $\pi: \mathcal{X} \to \mathbb{P}^1$, whose general fiber is a smooth curve of genus one, and a section $\sigma: \mathbb{P}^1 \to \mathcal{X}$ in the elliptic fibration $\pi$. If $\sigma'$ is another section of the Jacobian fibration $(\pi,\sigma)$, then there exists an automorphism of $\mathcal{X}$ preserving $\pi$ and mapping $\sigma$  to $\sigma'$. This automorphism can be constructed using the group law of the elliptic fiber. One can then realize an identification between the set of sections of $\pi$ and the group of automorphisms of $\mathcal{X}$ preserving $\pi$. This is the \emph{Mordell-Weil group} $\operatorname{MW}(\pi,\sigma)$ of the Jacobian fibration.  More precisely, this identification is unique up to finite index. In fact, if the $j$-invariant is non-constant then the Mordell-Weil group has index 2 in the group of automorphisms fixing $\pi$ since fiberwise multiplication by $-1$ is an automorphism preserving both $\pi$ and $\sigma$. If the $j$-invariant is constant then this index might be 2, 4 or 6. However, all Jacobian elliptic fibrations considered in this article are not isotrivial.
We have the following:
\begin{lemma}
\label{lem:lattice}
Let $\mathcal{X}$ be a general $L$-polarized K3 surface and $(\pi,\sigma)$ a Jacobian elliptic fibration on $\mathcal{X}$. Then, the Mordell-Weil group has finite order. In particular, we have
\beq
 \operatorname{rank} \operatorname{MW}(\pi,\sigma) = 0 \,.
\eeq
\end{lemma} 
\begin{proof}
For $\operatorname{NS}(\mathcal{X})=L$, it follows, via work of Nikulin \cites{MR544937,MR633160,MR752938} and Kondo \cite{MR1029967}, that the group of automorphisms of $\mathcal{X}$ is finite. In fact, $\operatorname{Aut}(\mathcal{X}) \simeq \mathbb{Z}_2 \times \mathbb{Z}_2$.  In particular, any Jacobian elliptic fibration on $\mathcal{X}$ must have a Mordell-Weil group of finite order and cannot admit any infinite-order section.
\end{proof}
Given a Jacobian elliptic fibration $(\pi,\sigma)$ on $\mathcal{X}$, the classes of fiber and section span a rank-two primitive sub-lattice of $\operatorname{NS}(\mathcal{X})$ which is isomorphic to the standard rank-two hyperbolic lattice $H$. The converse also holds: given a primitive lattice embedding $ j \colon H \hookrightarrow \operatorname{NS}(\mathcal{X})$, whose image $j(H)$ contains a pseudo-ample class, it is known (see  \cite{MR2355598}*{Thm.~2.3}) that there exists a Jacobian elliptic fibration on the surface $\mathcal{X}$, whose fiber and section classes span $j(H)$. Moreover, one has a one-to-one correspondence between classes of Jacobian elliptic fibrations on $\mathcal{X}$, up to automorphisms of $\mathcal{X}$, and classes of primitive lattice embeddings $H \hookrightarrow \operatorname{NS}(\mathcal{X})$, up to the action of isometries of $H^2(\mathcal{X}, \mathbb{Z})$ preserving the Hodge decomposition (see \cite{MR2369941}*{Lemma~3.8}). 
\par Assume $j \colon H \hookrightarrow L$ is a primitive embedding. Denote by $K $ the orthogonal complement of $j(H)$ in $L$. It follows that $L = j(H) \oplus K$. Moreover, the lattice $K$ is negative-definite of rank fourteen and its discriminant group must be of the form
\beq
\label{discr1}
 \Big( D( K ) , \; q_K  \Big) \ \simeq \ \Big( D(L), \; q_L \Big) \ 
  \simeq \ \Big( \mathbb{Z}_2 \oplus \mathbb{Z}_2 , \ (3/2) \oplus (3/2)  \Big)  \,.
\eeq
Denote by $K^{\text{root}}$ the sub-lattice spanned by the roots of $K$, i.e., the lattice elements of self-intersection $-2$  in $K$.  Let $\Sigma \subset \mathbb{P}^1$ be the set of points on the base of the elliptic fibration $\pi$ that correspond to singular fibers.  For each singular point $p \in \Sigma$, we denote by $\mathrm{T}_p$ the sub-lattice spanned by the classes of the irreducible components of the singular fiber over $p$ that are disjoint from the section $\sigma$ of the elliptic fibration. Standard K3 geometry arguments tell us that $K^{\text{root}}$ is of ADE-type, meaning for each $p \in \Sigma$ the lattice $\mathrm{T}_p$ is a negative definite lattice of type $A_m$, $D_m$ and $E_l$, and we have
\beq
 K^{\text{root}} = \bigoplus_{p \in \Sigma}  \mathrm{T}_p \,.
\eeq 
Another item relevant to our discussion is the factor group:
\beq
   \mathpzc{W} = K / K^{\text{root}}  \,.
\eeq
Shioda \cite{MR1081832} proved that there is a canonical group isomorphism $ \mathpzc{W} \simeq \operatorname{MW}(\pi,\sigma)$, identifying $ \mathpzc{W}$ with the Mordell-Weil group of the corresponding Jacobian elliptic fibration $(\pi,\sigma)$. 
\par As a first step in our classification of Jacobian elliptic fibrations on $\mathcal{X}$, we have: 
\begin{lemma}
\label{prop:ADE}
In the above context, one has $ \vert  \mathpzc{W} \vert \leqslant 2$ and the possible choices for $K^{\text{root}} $ are as follows:
\begin{itemize}
\item [(a)] If $ \  \mathpzc{W} = \{ \mathbb{I} \} $, then $ K^{\text{root}} $ is isomorphic to either 
$$E_7(-1) \oplus E_7(-1), \ \ E_8(-1) \oplus D_6(-1) , \ {\rm or} \ D_{14}(-1) $$ 
\item [(b)] If  $ \  \mathpzc{W} =\mathbb{Z}_2 $, then $K^{\text{root}} =D_{12}(-1) \oplus A_1(-1) \oplus A_1(-1)$ 
\end{itemize} 
\end{lemma}
\begin{proof}
A classification of  elliptic fibrations on K3 surfaces with 2-elementary N\'eron-Severi lattice was given in \cites{MR4130832, MR3201823}.  Based on Nikulin's classification \cite{MR633160b} and Shimada's result \cite{MR1813537}, a purely lattice theoretic classification of Jacobian elliptic fibrations with finite automorphism group was given in \cite{Clingher:2021}. Restricting to Picard number 16, the result follows.
\end{proof}
\par Let us also investigate the possible distinct primitive lattice embeddings $H \hookrightarrow L$. We shall employ the point of view  of \cite{FestiVeniani20}. Assume that $\operatorname{NS}(\mathcal{X})=L$ and $j' \colon H \hookrightarrow L$ is a second primitive embedding, such that the orthogonal complement of the image $j'(H)$, denoted $K'$, is isomorphic to the lattice $K$ above. We would like to see under what conditions $j$ and $j'$ correspond to Jacobian elliptic fibrations isomorphic under $\mathrm{Aut}(\mathcal{X}) $. By standard lattice-theoretic arguments (see  \cite{MR525944}*{Prop.~ 1.15.1}), there will exist an isometry $ \gamma \in \mathcal{O}(L) $ such that $j' = \gamma \circ j$. The isometry $\gamma$ has a counterpart $ \gamma^* \in \mathcal{O}(D(K)) $ obtained as image of $\gamma$ under the group homomorphism
\beq 
\label{izo22}
\mathcal{O}(L) \ \rightarrow \  
\mathcal{O}(D(L)) \ \simeq \ \mathcal{O}(D(K)) \,.
\eeq
The isomorphism in (\ref{izo22}) is due to the decomposition $L = j(H) \oplus K$ and, as such, it depends on the lattice embedding $j$. 
\par Denote the group $\mathcal{O}(D(K)) $ by $\mathcal{A}$. There are two subgroups of $\mathcal{A}$ that are relevant to our discussion. The first subgroup $\mathcal{B} \leqslant \mathcal{A}$  is given as the image of the following group homomorphism:
\beq
\mathcal{O}(K) \ \simeq \  
\left \{ 
\varphi \in \mathcal{O}(L) \ \vert \  \varphi \circ j(H) = j(H) 
\right \} 
 \ \hookrightarrow \ \mathcal{O}(L) \ \rightarrow \ \mathcal{O}(D(L)) \ \simeq \ \mathcal{O}(D(K))  \ .
\eeq
The second subgroup $\mathcal{C} \leqslant \mathcal{A}$ is obtained as the image of following group homomorphism:
\beq
\mathcal{O}_h(\mathrm{T}_{\mathcal{X}}) \ \hookrightarrow \  \mathcal{O}(\mathrm{T}_{\mathcal{X}}) 
\ \rightarrow \ \mathcal{O} ( D ( \mathrm{T}_{\mathcal{X}} ) )  \ \simeq \ 
 \mathcal{O}(D(L)) \ \simeq \ \mathcal{O}(D(K))  \ .
\eeq
Here $\mathrm{T}_{\mathcal{X}} $ denotes the transcendental lattice of the K3 surface $\mathcal{X}$ and $\mathcal{O}_h(\mathrm{T}_{\mathcal{X}})$ is given by the isometries of 
$\mathrm{T}_{\mathcal{X}} $ that preserve the Hodge decomposition. Furthermore, one has $ D(L) \simeq D(\mathrm{T}_{\mathcal{X}})  $  with $q_{L} = -q_{\mathrm{T}_{\mathcal{X}}} $, as $\operatorname{NS}(\mathcal{X})=L$ and $\mathrm{T}_{\mathcal{X}}$ is the orthogonal complement of $\operatorname{NS}(\mathcal{X}) $ with respect to an unimodular lattice. 
\par Consider then the correspondence
\beq
\label{corresp33} 
 H \stackrel{j}{\hookrightarrow} L  \quad  \rightsquigarrow \quad \mathcal{C} \gamma^* \mathcal{B} \,,
 \eeq
 that associates to a lattice embedding $H \hookrightarrow L$ a double-coset in $\mathcal{C} \backslash \mathcal{A}/  \mathcal{B}$. As proved in  \cite{FestiVeniani20}*{Thm~2.8},   the map (\ref{corresp33}) establishes a one-to-one correspondence between Jacobian elliptic fibrations on $\mathcal{X}$ with $ j(H)^{\perp} \simeq K $, up to the action of the automorphism group ${\rm Aut}(\mathcal{X}) $ and the elements of the double-coset set $\mathcal{C} \backslash \mathcal{A}/  \mathcal{B}$. The number of elements in $\mathcal{C} \backslash \mathcal{A}/  \mathcal{B}$ is referred by Festi and Veniani as the {\it multiplicity} of the lattice $K$.  
Following up on Lemma~\ref{prop:ADE}, one then obtains:
\begin{proposition} 
\label{cor:EFS} 
\par \noindent
\begin{itemize}
\item [(a)] Up to a lattice isomorphism, there are exactly four rank-fourteen negative-definite lattices $K$, satisfying condition~(\ref{discr1}). We list those as 
$K_i, \ 1 \leq i \leq 4 $, where:
$$  K_1  = E_7(-1) \oplus E_7(-1) \,, \quad   K_2 = E_8(-1) \oplus D_6(-1) \,, \quad   K_3 = D_{14}(-1) \ , $$
while $ K_4 $ is an over-lattice of $D_{12}(-1) \oplus A_1(-1) \oplus A_1(-1)$, of index two. 
\item [(b)] For each of the four possible choices in (a), the multiplicity of $K_i$ is $1$.    
\end{itemize} 
\end{proposition}
\begin{proof}
The statement is proven by a computation using the Sage class {\tt QuadraticForm}.  For a given lattice {\tt K}, the discriminant group is computed using the command {\tt D = K.discriminant\char`_group()}. The automorphism groups are computed using
\beqn
\text{{\tt O=K.orthogonal\char`_group()} and {\tt C=D.orthogonal\char`_group()}.}
\eeqn
Images of the generators are computed as {\tt A=D.orthogonal\char`_group(O.gens())}.
\par For $\mathpzc{W}=\mathbb{Z}/2\mathbb{Z}$  the lattice $K_4$ is the overlattice spanned by $K^{\text{root}}$ and one additional lattice vector $\vec{v}_{\text{max}}$. The vector $\vec{v}_{\text{max}}$ is in the orthogonal complement of the sublattice spanned by the section and the smooth fiber class. The lattice vector $\vec{v}_{\text{max}}$ can be computed using the properties of the elliptic fibration alone. Using the same notation for the bases of root lattices as in \cite{MR1813537}*{Sec.~6} and ordering the bases by ADE-type and from lowest to highest rank, one obtains
\beqn
  \vec{v}_{\text{max}} = \frac{1}{2} \, \langle  1, 1  \, | \, 1, 0,  1, 0, 1, 0, 1, 0, 1, 0, 1 , 0 \rangle \,.
\eeqn
The overlattice is computed using the command {\tt overlattice}. The corresponding Gram matrix is then computed using the command {\tt gram\char`_matrix}.  
\par  The multiplicity associated with each $(K^{\text{root}}, \mathpzc{W})$ is then shown to equal one by checking that the images of the generators for $O(K)$ also generate $O(D(K))$.
\end{proof}
\noindent Proposition \ref{cor:EFS} implies the following corollaries:
\begin{corollary}
\label{cor:ellipt}
A general $L$-polarized K3 surface $\mathcal{X}$ admits exactly four inequivalent Jacobian elliptic fibrations $(\pi,\sigma)$, which are unique up to the action of the automorphism group ${\rm Aut}(\mathcal{X})$.
\end{corollary}
\begin{corollary}
\label{cor:equiv}
One has four isometric manifestations of $L$:
\beq
\label{eqn:2elementary_lattices}
H \oplus E_7(-1) \oplus E_7(-1)  \ \cong \ 
  \ H \oplus E_8(-1)  \oplus D_6(-1) \ \cong \  \ H \oplus D_{14}(-1) \ \cong \ H \oplus K_4 \,.
\eeq
\end{corollary}
\section{Generalized Inose quartic and its elliptic fibrations}
\label{sec:geometry}
In \cites{MR2369941,MR2279280} it was proved that  a complex algebraic K3 surface $\mathcal{X}$ with Picard lattice $H \oplus E_8(-1) \oplus E_8(-1)$ admits a birational model isomorphic to the quartic surface in $\mathbb{P}^3=\mathbb{P}(\mathbf{X}, \mathbf{Y}, \mathbf{Z}, \mathbf{W})$ with equation
\beqn
 0= \  \mathbf{Y}^2 \mathbf{Z} \mathbf{W}-4 \mathbf{X}^3 \mathbf{Z}+3 \alpha \mathbf{X} \mathbf{Z} \mathbf{W}^2+ \beta \mathbf{Z} \mathbf{W}^3 
-   \frac{1}{2} \big( \mathbf{Z}^2  \mathbf{W}^2 +  \mathbf{W}^4 \big) .
\eeqn
The 2-parameter family was first introduced by Inose in \cite{MR578868} and is called \emph{Inose quartic}. Other examples of equations relating the elliptic fibrations of K3 surfaces with 2-elementary N\'eron-Severi lattice and quartic hypersurfaces in $\mathbb{P}^3$ were provided in \cites{MR4130832, MR3882710}. We will consider a multi-parameter generalizations of the Inose quartic.
\par Let $(\alpha, \beta, \gamma, \delta , \varepsilon, \zeta) \in \mathbb{C}^6 $ be a set of parameters. We consider the projective quartic surface $ \mathcal{Q}(\alpha, \beta, \gamma, \delta , \varepsilon, \zeta)$ in $\mathbb{P}^3$ defined by the homogeneous equation
\beq
\label{mainquartic}
\mathbf{Y}^2 \mathbf{Z} \mathbf{W}-4 \mathbf{X}^3 \mathbf{Z}+3 \alpha \mathbf{X} \mathbf{Z} \mathbf{W}^2+\beta \mathbf{Z} \mathbf{W}^3+\gamma \mathbf{X} \mathbf{Z}^2 \mathbf{W}- \frac{1}{2} \left (\delta \mathbf{Z}^2 \mathbf{W}^2+ \zeta \mathbf{W}^4 \right )+ \varepsilon \mathbf{X} \mathbf{W}^3 = 0.
\eeq
The family~(\ref{mainquartic}) was first introduced by the first author and Doran in \cite{MR2824841}. We denote by $\mathcal{X}(\alpha, \beta, \gamma, \delta , \varepsilon, \zeta)$ the smooth complex surface obtained as the minimal resolution of  $\mathcal{Q}(\alpha, \beta, \gamma, \delta , \varepsilon, \zeta)$. The quartic $\mathcal{Q}(\alpha, \beta, \gamma, \delta , \varepsilon, \zeta)$ has two special singularities at the following points:
\beq
  \mathrm{P}_1 = [0,1,0,0], \ \ \ \ \mathrm{P}_2 = [0,0,1,0] \,.
\eeq  
One verifies that the singularity at $\mathrm{P}_1$ is a rational double point of type $A_9$  if $ \varepsilon \neq 0$, and of type $A_{11}$ if $ \varepsilon = 0$. The singularity at $\mathrm{P}_2$ is of type $A_5$  if $ \gamma  \neq 0$, and of type $E_6$ if $ \gamma  = 0$. One easily checks that for $ (\gamma, \delta , \varepsilon, \zeta) \neq 0$,  the points $\mathrm{P}_1$ and $\mathrm{P}_2$ are the only singularities of Equation~(\ref{mainquartic}) and are rational double points. The two sets $ a_1, a_2, \dots, a_9$ and $b_1, b_2, \dots, b_5$ will denote the curves appearing from resolving the rational double  point singularities at $\mathrm{P}_1$ and $\mathrm{P}_2$, respectively. 
\par In this section we assume that $\gamma \varepsilon \not =0$. The specializations $\gamma=0$ and $\varepsilon=0$ were already considered in \cite{MR4015343}. We introduce the following three special lines, denoted by $ L_1 $, $ L_2 $, $L_3$ and given by
\beqn
 \mathbf{X}=\mathbf{W}=0 \,, \quad \mathbf{Z}=\mathbf{W}=0\,, \quad 2\varepsilon \mathbf{X}-\zeta \mathbf{W} = \mathbf{Z} = 0 \,. 
\eeqn 
Note that $ L_1 $, $ L_2 $, $L_3$ lie on the quartic in Equation~(\ref{mainquartic}). Because of $ \gamma \varepsilon \neq 0 $, the lines $ L_1 $, $ L_2 $, $L_3$  are distinct and concurrent, meeting at $\mathrm{P}_1$. Moreover, we consider the following complete intersections:
\beqn
 2\varepsilon \mathbf{X}-\zeta \mathbf{W}  \ = \  
\left (3 \alpha \varepsilon ^2 \zeta + 2 \beta \varepsilon ^3 - \zeta^3 \right ) \mathbf{W}^2 -  
\varepsilon^2 \left ( \delta \varepsilon-\gamma \zeta \right ) \mathbf{Z} \mathbf{W} 
+ 2 \varepsilon^3 \mathbf{Y}^2 \ = \  0 \,,
\eeqn
\beqn
2\gamma \mathbf{X}-\delta \mathbf{W}  \ = \  
\left (3 \alpha \gamma ^2 \delta + 2 \beta \gamma ^3 - \delta^3 \right ) \mathbf{Z} \mathbf{W}^2 -  
\gamma^2 \left ( \gamma \zeta - \delta \varepsilon \right ) \mathbf{W}^3 
+ 2 \gamma^3 \mathbf{Y}^2 \mathbf{Z} \ = \  0 \, .
\eeqn
Assuming appropriate generic conditions, the above equations determine two projective curves $R_1$, $R_2$, of degrees two and three, respectively. 
The conic $R_1$ is a (generically smooth) rational curve tangent to $L_1$ at $\mathrm{P}_2$. The cubic $R_2$ has 
a double point at $\mathrm{P}_2$, passes through $\mathrm{P}_1$ and is generically irreducible. When resolving the quartic surface $(\ref{mainquartic})$, 
these two curves lift to smooth rational curves on $\mathcal{X}(\alpha, \beta, \gamma, \delta , \varepsilon, \zeta)$, which by a slight 
abuse of notation we shall denote by the same symbol.  
\par We define the \emph{dual graph} of smooth rational curves to be the simplicial complex whose set of vertices is the set of smooth rational curves on a K3 surface such that the vertices $\Sigma, \Sigma'$ are joint by an $m$-fold edge if and only if their intersection product is $\Sigma  \cdot \Sigma' = m$.  For Picard rank bigger than or equal to $15$, the possible dual graphs of all smooth rational curves on K3 surfaces with finite automorphism groups were determined in  \cite{MR556762}.  Here, we use the dual graph for $H \oplus E_8(-1) \oplus D_6(-1)$-polarized surfaces  in  \cite{MR556762}*{Sec.~\!4} and the fact that, due to Corollary~\ref{cor:equiv}, it is identical with the graph for an $H \oplus E_7(-1) \oplus E_7(-1)$-polarization. The dual diagram of rational curves is given by the following graph:
\begin{equation}
\label{exdiagg77}
\def\objectstyle{\scriptstyle}
\def\labelstyle{\scriptstyle}
\scalemath{\MyScaleBig}{
\xymatrix @-0.9pc  {
& & \stackrel{L_3}{\bullet} \ar @{=}[rrrrrrrrrr]& &  & & &    & & & & & \stackrel{R_1}{\bullet} &  &  \\ 
 & &  & &  & & & &   &    & &  & & \\
& \stackrel{a_1}{\bullet} \ar @{-} [r] \ar @{-}[ddr] \ar @{-}[uur] &
\stackrel{a_2}{\bullet} \ar @{-} [r]  &
\stackrel{a_3}{\bullet} \ar @{-} [r]  \ar @{-} [d] &
\stackrel{a_4}{\bullet} \ar @{-} [r] &
\stackrel{a_5}{\bullet} \ar @{-} [r] &
\stackrel{a_6}{\bullet} \ar @{-} [r] &
\stackrel{a_7}{\bullet} \ar @{-} [r] &
\stackrel{a_8}{\bullet} \ar @{-} [r] &
\stackrel{a_9}{\bullet} \ar @{-} [r] &
\stackrel{L_1}{\bullet} \ar @{-} [r] &
\stackrel{b_2}{\bullet} \ar @{-} [r] \ar @{-} [d] &
\stackrel{b_3}{\bullet} & \stackrel{b_4}{\bullet} \ar @{-} [l] \ar @{-}[ddl] \ar @{-}[uul]
 & \\
 & &  & \stackrel{L_2}{\bullet} &  & & & &   &    & & \stackrel{b_1}{\bullet}  & & \\
 & & \stackrel{R_2}{\bullet} \ar @{=}[rrrrrrrrrr]& &  & & &    & & & & &   \stackrel{b_5}{\bullet} & & \\
} }
\end{equation}
We proved the following theorem in \cite{MR4015343}:
\begin{theorem}
\label{thm:polarization}
Assume that $(\gamma, \delta) \neq (0,0)$ and $ (\varepsilon, \zeta) \neq (0,0)$. Then, the surface $\mathcal{X}(\alpha, \beta, \gamma, \delta , \varepsilon, \zeta)$ obtained as the minimal resolution of $\mathcal{Q}(\alpha, \beta, \gamma, \delta , \varepsilon, \zeta)$ is a K3 surface endowed with a canonical $L$-polarization. 
\end{theorem}
\par A simple computation shows:
\begin{lemma}
The degree-four polarization determined on $\mathcal{X}(\alpha, \beta, \gamma, \delta , \varepsilon, \zeta)$, assuming the case  
$ \gamma \varepsilon \neq 0 $, is given by the following polarizing divisor
\beq
\label{linepolariz}
\mathcal{H} =  L_2 + 
\left ( 
a_1+2a_2+3a_3+3a_4+3a_5+\dots + 3a_9 
\right ) 
+ 3 L_1 + 
\left ( 
2b_1+4b_2+3b_3+2b_4+b_5
\right ).
\eeq
\end{lemma}
\par The following lemma was proved in \cite{MR4015343}:
\begin{lemma}
\label{symmetries1}
Let $(\alpha, \beta, \gamma, \delta, \varepsilon, \zeta) \in \mathbb{C}^6 $ with $(\gamma, \delta) \neq (0,0)$ and $(\varepsilon, \zeta)  \neq (0,0)$. Then, one has the following isomorphisms of $L$-polarized K3 surfaces:
\begin{enumerate}[label=(\alph*)]
\item  $\mathcal{X}(\alpha,\beta, \gamma, \delta, \varepsilon, \zeta) \ \simeq \
\mathcal{X}(t^2 \alpha,  \ t^3 \beta,  \ t^5 \gamma,  \ t^6 \delta,  \ t^{-1} \varepsilon,  \ \zeta  ) $, for any $t \in \mathbb{C}^*$,
\item $\mathcal{X}(\alpha,\beta, \gamma, \delta, \varepsilon, \zeta) \ \simeq \  
\mathcal{X}(\alpha,  \beta,  \varepsilon,  \zeta, \gamma,  \delta  )$
\end{enumerate}
\end{lemma}
\par We also ave the following:
\begin{proposition}
\label{NikulinInvolution}
Let $(\alpha, \beta, \gamma, \delta, \varepsilon, \zeta) \in \mathbb{C}^6 $ as before. A Nikulin involution on the $L$-polarized K3 surfaces $\mathcal{X}(\alpha, \beta, \gamma, \delta , \varepsilon, \zeta)$ is induced by the map
\beq
\label{eq:NikulinInvolutions}
\begin{split}
 \Psi:\qquad  \mathbb{P}^3 \  \longrightarrow & \ \mathbb{P}^3, \\
  [\mathbf{X}: \mathbf{Y}: \mathbf{Z}: \mathbf{W}] \  \mapsto & \ [  \ (2 \gamma \mathbf{X} - \delta \mathbf{W}) \mathbf{X}\mathbf{Z} \  : \ - (2 \gamma \mathbf{X} - \delta \mathbf{W}) \mathbf{Y}\mathbf{Z}   :\\
  & \quad \, (2 \varepsilon \mathbf{X} - \zeta \mathbf{W})\mathbf{W}^2 \, : \ (2 \gamma \mathbf{X} - \delta \mathbf{W}) \mathbf{Z}^2 \ ] \,.
\end{split}
\eeq
\end{proposition}
\begin{proof}
One checks that $\Psi$ constitutes an involution of the projective quartic surface $\mathcal{Q}(\alpha, \beta, \gamma, \delta , \varepsilon, \zeta) \subset  \mathbb{P}^3(\mathbf{X}, \mathbf{Y}, \mathbf{Z}, \mathbf{W})$.  If we use the affine chart $\mathbf{W}=1$ then the unique holomorphic 2-form is given by $d\mathbf{X} \wedge d\mathbf{Y} / F_\mathbf{Z} ( \mathbf{X},\mathbf{Y}, \mathbf{Z})$ where $F( \mathbf{X},\mathbf{Y}, \mathbf{Z})$ is the left hand side of Equation~(\ref{mainquartic}). One then checks that Equation~(\ref{eq:NikulinInvolutions}) constitutes a symplectic involution after using $F( \mathbf{X},\mathbf{Y}, \mathbf{Z})=0$. 
\end{proof}
\begin{remark}
For the K3 surfaces $\mathcal{X}(\alpha, \beta, \gamma, \delta , \varepsilon, \zeta)$, quotienting by the involution and blowing up recovers a double sextic surface, i.e., the double cover of the projective plane branched along the union of six lines. The corresponding family of double sextic surfaces was described in \cites{MR2254405,MR4015343}.
\end{remark}
We now state a major result of this article:
\begin{theorem}
\label{thm1}
The minimal resolution of the quartic surface~(\ref{mainquartic}) is a K3 surface $\mathcal{X}$ endowed with a canonical $L$-polarization. Conversely, a general $L$-polarized K3 surface $\mathcal{X}$ has a birational projective model~(\ref{mainquartic}). In particular, every Jacobian elliptic fibration on $\mathcal{X}$ is attained as a pencil of lines or conics as follows:
\begin{center}
\scalebox{\MyScaleMed}{
\begin{tabular}{l|l|l|l|l}
name 		& singular fibers & $\operatorname{MW}(\pi,\sigma)$  & pencil & Eqn.\\
\hline
standard 		& $2 III^* + 6 I_1$ 		& trivial 					&		
$\begin{array}{l} 
\text{residual surface intersection} \\ 
\text{of $L_2(u, v)=0$ and  $\mathcal{Q}(\alpha, \beta, \gamma, \delta , \varepsilon, \zeta)$} \end{array}$ & (\ref{eqn:std}) \\
\hline
alternate 		& $I_8^* + 2 I_2 + 6 I_1$ 	& $\mathbb{Z}/2\mathbb{Z}$ 	& 
$\begin{array}{l} 
\text{residual surface intersection} \\ 
\text{of $L_1(u, v)=0$ and  $\mathcal{Q}(\alpha, \beta, \gamma, \delta , \varepsilon, \zeta)$} \end{array}$ & (\ref{eqn:alt})\\
\hline
base-fiber dual 	& $II^* + I_2^* + 6 I_1$ 	& trivial 					& 
$\begin{array}{l} 
\text{residual surface intersection} \\ 
\text{of $L_3(u, v)=0$ and  $\mathcal{Q}(\alpha, \beta, \gamma, \delta , \varepsilon, \zeta)$} \end{array}$ & (\ref{eqn:bfd})\\
\hline
maximal 		& $I_{10}^* + 8 I_1$ 		& trivial					& 
$\begin{array}{l} 
\text{residual surface intersection} \\ 
\text{of $C_3(u, v)=0$ and  $\mathcal{Q}(\alpha, \beta, \gamma, \delta , \varepsilon, \zeta)$} \end{array}$ & (\ref{eqn:max}) \\
\end{tabular}}
\end{center}
\end{theorem}
\begin{proof}
We construct the Weierstrass models for these four Jacobian elliptic fibrations explicitly in the Sections~\ref{ssec:std}-\ref{ssec:max}. The fact that these Jacobian elliptic fibrations are the \emph{only} possible fibrations was already proven in Proposition~\ref{cor:EFS}. Conversely, Proposition~\ref{cor:EFS} and Lemma~\ref{prop:ADE} prove that every general $L$-polarized K3 surface admits a unique fibration with the singular fibers $ 6 I_1 + 2 I_2 + I_8^*$ and a Mordell-Weil group $\mathbb{Z}/2\mathbb{Z}$ that can be brought into the form of Equation~(\ref{eqn:alt}). It follows from Equations~(\ref{eqn:substitution_alt}) that from such a fibration a quartic surface can be constructed if we write the polynomials $A$ and $B$ according to Equation~(\ref{eqn:pf_comparison}).  
\end{proof}
\subsection{The standard fibration}
\label{ssec:std}
There are exactly two ways of embedding two disjoint reducible fibers, each given by an extended Dynkin diagram $\widetilde{E}_7$, into the diagram~(\ref{exdiagg77}). These two ways are depicted in Figure~\ref{fig:e7e7fib}, where the green and the blue nodes indicate the two reducible fibers. In the case of Figure~\ref{Fig1a}, we have
\beq
\begin{split}
{\color{green}\widetilde{E}_7}=   \langle L_3,  a_1,  a_2, a_3, L_2,  a_4,  a_5 ; a_6 \rangle \,, \quad
{\color{blue}\widetilde{E}_7} =   \langle b_5,  b_4,  b_3, b_2, b_1, L_1, a_9 ; a_8 \rangle \, .
\end{split}
\eeq
Thus, the smooth fiber class is given by
\beq
 \mathrm{F}^{(a)}_{\text{std}} = L_3 + 2 a_1 + 3 a_2 + 4 a_3 + 2 L_2 + 3  a_4 + 2 a_5 +  a_6  \,,
 \eeq
and the class of a section is ${\color{red} a_7}$. 
\begin{figure}
\hspace*{-1.5cm}
\begin{subfigure}{.5\textwidth}  
\centering
\begin{equation*}
\scalemath{\MyScaleSmall}{
\xymatrix @-0.9pc  {
& & \color{green}{\stackrel{L_3}{\bullet}} \ar @{=}[rrrrrrrrrr]& &  & & &    & & & & & \stackrel{R_1}{\bullet} &  &  \\ 
 & &  & &  & & & &   &    & &  & & \\
&  
 \color{green}{\stackrel{a_1}{\bullet}} \ar @{-} [r] \ar @{-}[ddr] \ar @{-}[uur] &
 \color{green}{\stackrel{a_2}{\bullet}} \ar @{-} [r]  &
 \color{green}{\stackrel{a_3}{\bullet}} \ar @{-} [r]  \ar @{-} [d] &
 \color{green}{\stackrel{a_4}{\bullet}} \ar @{-} [r] &
 \color{green}{\stackrel{a_5}{\bullet}} \ar @{-} [r] &
 \color{green}{\stackrel{a_6}{\bullet}} \ar @{-} [r] &
 \color{red}{\stackrel{a_7}{\bullet}} \ar @{-} [r] &
 \color{blue}{\stackrel{a_8}{\bullet}} \ar @{-} [r] &
 \color{blue}{\stackrel{a_9}{\bullet}} \ar @{-} [r] &
 \color{blue}{\stackrel{L_1}{\bullet}} \ar @{-} [r] &
 \color{blue}{\stackrel{b_2}{\bullet}} \ar @{-} [r] \ar @{-} [d] &
 \color{blue}{\stackrel{b_3}{\bullet}} & \ \color{blue}{\stackrel{b_4}{\bullet}} \ar @{-} [l] \ar @{-}[ddl] \ar @{-}[uul]
 & \\
 & &  &  \color{green}{\stackrel{L_2}{\bullet}} &  & & & &   &    & &  \color{blue}{\stackrel{b_1}{\bullet}}  & & \\
 & & \stackrel{R_2}{\bullet} \ar @{=}[rrrrrrrrrr]& &  & & &    & & & & &    \color{blue}{\stackrel{b_5}{\bullet}} & & \\
 } }
\end{equation*}
\caption{\phantom{A}}
\label{Fig1a}
\end{subfigure}%
\quad
\begin{subfigure}{.5\textwidth}
\centering%
\begin{equation*}
\scalemath{\MyScaleSmall}{
\xymatrix @-0.9pc  {
& & \stackrel{L_3}{\bullet} \ar @{=}[rrrrrrrrrr]& &  & & &    & & & & & \color{green}{\stackrel{R_1}{\bullet}} &  &  \\ 
 & &  & &  & & & &   &    & &  & & \\
&  
 \color{blue}{\stackrel{a_1}{\bullet}} \ar @{-} [r] \ar @{-}[ddr] \ar @{-}[uur] &
 \color{blue}{\stackrel{a_2}{\bullet}} \ar @{-} [r]  &
 \color{blue}{\stackrel{a_3}{\bullet}} \ar @{-} [r]  \ar @{-} [d] &
 \color{blue}{\stackrel{a_4}{\bullet}} \ar @{-} [r] &
 \color{blue}{\stackrel{a_5}{\bullet}} \ar @{-} [r] &
 \color{blue}{\stackrel{a_6}{\bullet}} \ar @{-} [r] &
 \color{red}{\stackrel{a_7}{\bullet}} \ar @{-} [r] &
 \color{green}{\stackrel{a_8}{\bullet}} \ar @{-} [r] &
 \color{green}{\stackrel{a_9}{\bullet}} \ar @{-} [r] &
 \color{green}{\stackrel{L_1}{\bullet}} \ar @{-} [r] &
 \color{green}{\stackrel{b_2}{\bullet}} \ar @{-} [r] \ar @{-} [d] &
 \color{green}{\stackrel{b_3}{\bullet}} & \ \color{green}{\stackrel{b_4}{\bullet}} \ar @{-} [l] \ar @{-}[ddl] \ar @{-}[uul]
 & \\
 & &  &  \color{blue}{\stackrel{L_2}{\bullet}} &  & & & &   &    & &  \color{green}{\stackrel{b_1}{\bullet}}  & & \\
 & & \color{blue}{\stackrel{R_2}{\bullet}} \ar @{=}[rrrrrrrrrr]& &  & & &    & & & & &    \stackrel{b_5}{\bullet} & & \\
 } }
\end{equation*}
\caption{\phantom{B}}
\label{Fig1b}
\end{subfigure}%
\caption{The \emph{standard fibrations} with 2 fibers of type $\widetilde{E}_7$}
\label{fig:e7e7fib}
\end{figure}%
Making the substitutions  
\beq
 \mathbf{X} = u v x\,, \quad  \mathbf{Y}= y\,, \quad \mathbf{Z} = 4 u^4 v^2 z\,, \quad \mathbf{W} = 4 u^3 v^3 z \,,
\eeq
in Equation~(\ref{mainquartic}), is compatible with $L_2(u, v)=0$. Here, $L_2(u, v)=u \mathbf{W} - v \mathbf{Z} =0$ for $[u:v] \in \mathbb{P}^1$ is the pencil of planes containing the line $L_2$. One obtains  the Jacobian elliptic fibration $\pi_{\text{std}} : \mathcal{X} \rightarrow \mathbb{P}^1$ with fiber $\mathcal{X}_{[u:v]}$, given by the Weierstrass equation 
\beq
\label{eqn:std}
\mathcal{X}_{[u:v]}: \quad y^2 z  = x^3 +  f(u, v) \, x z^2 + g(u, v) \, z^3 \,, 
\eeq
admitting the section $\sigma_{\text{std}} : [x:y:z] = [0:1:0]$, and with the discriminant 
\beq
\Delta(u, v) =  4 f^3 + 27 g^2 =  64 \, u^9 v^9  p(u, v) \,,
\eeq
where
\beq
f(u, v)  = - 4 u^3 v^3  \Big(\gamma  u^2 + 3  \alpha  uv + \varepsilon v^2\Big) \,, \quad
g(u, v)  =  8  u^5 v^5 \Big(\delta  u^2 - 2  \beta  u v  + \zeta  v^2\Big) \,,
\eeq
and $p(u, v) = 4  \gamma^3 u^6 + \dots + 4 \varepsilon^3 v^6$ is an irreducible homogeneous polynomial of degree six.
We have the following:
\begin{lemma}
\label{lem:fib_std}
Equation~(\ref{eqn:std}) defines a Jacobian elliptic fibration with the singular fibers $6 I_1 + 2 III^*$ and a trivial Mordell-Weil group $\operatorname{MW}(\pi_{\text{std}}, \sigma_{\text{std}}) = \{ \mathbb{I}\}$.
\end{lemma}
\begin{proof}
The proof easily follows by checking the Kodaira type of the singular fibers at $p(u, v)=0$ and $u = 0$ and $v = 0$.
\end{proof}
\par Applying the Nikulin involution in Proposition~\ref{NikulinInvolution}, we obtain the fiber configuration in Figure~\ref{Fig1b} with
\beq
\begin{split}
{\color{blue}\widetilde{E}_7} =   \langle R_2,  a_1,  a_2,  a_3, L_2, a_4, a_5 ; a_6 \rangle \,, \quad
{\color{green}\widetilde{E}_7} =   \langle R_1,  b_4,  b_3,  b_2, b_1, L_1, a_9 ; a_8 \rangle \,.
\end{split}
\eeq
The smooth fiber class is given by
\beq
 \mathrm{F}^{(b)}_{\text{std}} =  R_2 + 2 a_1 + 3 a_2 + 4 a_3 + 2 L_2 + 3  a_4 + 2 a_5 +  a_6  \,,
 \eeq
and the class of the section is ${\color{red} a_7}$. Using the polarizing divisor $\mathcal{H}$ in Equation~(\ref{linepolariz}), one checks that
\beq
 2 \mathcal{H} -  \mathrm{F}^{(b)}_{\text{std}} - L_1  - L_2  - L_3\equiv 2 a_1 + 3 a_2 + 4 a_3 + \dots + 4  a_7 + 3 a_8 + 2 a_9 + b_1 + 2 b_2 + \dots + 2 b_5\,,
\eeq
which shows that the fibration is also induced by intersecting the quartic with the pencil of quadratic surfaces, denoted by $C_1(u, v)=0$ with $[u:v] \in \mathbb{P}^1$, containing the lines $L_1, L_2, L_3$. A computation yields
\beq
 C_1(u, v)= v \mathbf{W} \big( 2 \varepsilon \mathbf{X}- \zeta \mathbf{W} \big)- u \mathbf{Z} \big( 2 \gamma \mathbf{X} - \delta \mathbf{W}\big) =0 \,.
\eeq
\subsection{The alternate fibration}
\label{ssec:alt}
There is exactly one way of embedding a reducible fiber given by the extended Dynkin diagram $\widetilde{D}_{12}$ and two reducible fibers of type $\widetilde{A}_{1}$ into the diagram~(\ref{exdiagg77}). The configuration is invariant when applying the Nikulin involution in Proposition~\ref{NikulinInvolution} and shown in Figure~\ref{fig:d12fib}.
We have
\beq
\begin{split}
{\color{green}\widetilde{A}_1}=   \langle L_3 ; R_1  \rangle \,, \quad
{\color{blue}\widetilde{A}_1} =   \langle R_2 ;  b_5 \rangle \,, \quad {\color{brown}\widetilde{D}_{12}}=  \langle a_2 , L_2 ,  \dots , L_1, b_1 ; b_3  \rangle \,,
\end{split}
\eeq
Thus, the smooth fiber class is given by
\beq
 \mathrm{F}_{\text{alt}} = a_2+2 a_3 + L_2 +2 a_4 + \dots + 2 a_9 +2  L_1 +b_1 +2 b_2+ b_3 \,,
\eeq
and the classes of a section and 2-torsion section are ${\color{red} a_1, b_4}$. 
\begin{figure}
\centering
\begin{equation*}
\scalemath{\MyScaleMed}{
\xymatrix @-0.9pc  {
& &  \color{green}{\stackrel{L_3}{\bullet}} \ar @{=}[rrrrrrrrrr]& &  & & &    & & & & & \color{green}{\stackrel{R_1}{\bullet}} &  &  \\ 
 & &  & &  & & & &   &    & &  & & \\
&  
 \color{red}{\stackrel{a_1}{\bullet}} \ar @{-} [r] \ar @{-}[ddr] \ar @{-}[uur] &
 \color{brown}{\stackrel{a_2}{\bullet}} \ar @{-} [r]  &
 \color{brown}{\stackrel{a_3}{\bullet}} \ar @{-} [r]  \ar @{-} [d] &
 \color{brown}{\stackrel{a_4}{\bullet}} \ar @{-} [r] &
 \color{brown}{\stackrel{a_5}{\bullet}} \ar @{-} [r] &
 \color{brown}{\stackrel{a_6}{\bullet}} \ar @{-} [r] &
 \color{brown}{\stackrel{a_7}{\bullet}} \ar @{-} [r] &
 \color{brown}{\stackrel{a_8}{\bullet}} \ar @{-} [r] &
 \color{brown}{\stackrel{a_9}{\bullet}} \ar @{-} [r] &
 \color{brown}{\stackrel{L_1}{\bullet}} \ar @{-} [r] &
 \color{brown}{\stackrel{b_2}{\bullet}} \ar @{-} [r] \ar @{-} [d] &
 \color{brown}{\stackrel{b_3}{\bullet}} & \ \color{red}{\stackrel{b_4}{\bullet}} \ar @{-} [l] \ar @{-}[ddl] \ar @{-}[uul]
 & \\
 & &  &  \color{brown}{\stackrel{L_2}{\bullet}} &  & & & &   &    & &  \color{brown}{\stackrel{b_1}{\bullet}}  & & \\
 & & \color{blue}{\stackrel{R_2}{\bullet}} \ar @{=}[rrrrrrrrrr]& &  & & &    & & & & &    \color{blue}{\stackrel{b_5}{\bullet}} & & \\
 } }
\end{equation*}
\caption{The \emph{alternate fibration} with a fiber of type $\widetilde{D}_{12}$}
\label{fig:d12fib}
\end{figure}%
Making the substitutions 
\beq
\label{eqn:substitution_alt}
\mathbf{X}= 2 u v x\,, \quad  \mathbf{Y}=  y \,, \quad \mathbf{Z} = 4 v^5 (-2 \varepsilon u+\zeta v) z\,, \quad \mathbf{W} = 2 v^2 x\,,
\eeq
into Equation~(\ref{mainquartic}), is compatible with  $L_1(u, v)=0$. Here, $L_1(u, v)=u \mathbf{W} - v \mathbf{X} =0$ for $[u:v] \in \mathbb{P}^1$ is the pencil of planes containing the line $L_1$. We obtain the Jacobian elliptic fibration  $\pi_{\text{alt}}:  \mathcal{X} \rightarrow \mathbb{P}^1$ with fiber $\mathcal{X}_{[u:v]}$, given by the equation 
\beq
\label{eqn:alt}
\mathcal{X}_{[u:v]}: \quad y^2 z  = x \Big(x^2 + A(u, v) \, x z +B(u, v) \, z^2 \Big) \,, 
\eeq
admitting the section $\sigma_{\text{alt}} : [x:y:z] = [0:1:0]$ and the 2-torsion section  $[x:y:z] = [0:0:1]$, and with the discriminant 
\beq
\Delta(u, v) =  B(u, v)^2 \, \Big(A(u, v)^2-4 B(u, v)\Big) \,,
\eeq
where
\beq
\label{eqn:pf_comparison}
\begin{split}
A(u, v) = 4 v (4 u^3-3\alpha  uv^2- \beta v^3) \,,\quad
B(u, v) = 4 v^6 (2\gamma u-\delta v)(2\varepsilon u -\zeta v)   \,.
\end{split}
\eeq
We have the following:
\begin{lemma}
\label{lem:fib_alt}
Equation~(\ref{eqn:alt}) defines a Jacobian elliptic fibration with singular fibers $6 I_1 + 2 I_2 + I_8^*$ and a Mordell-Weil group $\operatorname{MW}(\pi_{\text{alt}}, \sigma_{\text{alt}}) = \mathbb{Z}/2\mathbb{Z}$. The Nikulin involution in Proposition~\ref{NikulinInvolution} acts on the Jacobian elliptic fibration~(\ref{eqn:alt}) as a van~Geemen-Sarti  involution \cite{MR2274533}.
\end{lemma}
\begin{proof}
The proof easily follows by checking the Kodaira type of the singular fibers at $B(u, v)=0$ and $A(u, v)^2-4 B(u, v)=0$. Applying the Nikulin involution in Proposition~\ref{NikulinInvolution}, we obtain the same configuration of  singular fibers, only with the roles of the section and the 2-torsion section interchanged. This means that the involution in Equation~(\ref{eq:NikulinInvolutions}) acts on the Jacobian elliptic fibration~(\ref{eqn:alt}) by fiberwise translation by 2-torsion, i.e., by mapping
\beq
 \Big[ x : y : z \Big] \mapsto \Big[ B(u, v) \, x z \  : - B(u, v) \, y z \ : \ x^2 \Big] 
\eeq 
for $[x:y:z] \not= [0:1:0], [0:0:1]$, and swapping $[0:1:0] \leftrightarrow [0:0:1]$.
\end{proof}
\subsection{The base-fiber dual fibration}
\label{ssec:bfd}
The K3 surfaces given by Equation~(\ref{eqn:alt}) are double covers of the Hirzebruch surface $\mathbb{F}_0=\mathbb{P}^1\times\mathbb{P}^1$ branched along a curve of type $(4,4)$, i.e., along a section in the line bundle $\mathcal{O}_{\mathbb{F}_0}(4,4)$. Every such cover has two natural elliptic fibrations corresponding to the two rulings of the quadric $\mathbb{F}_0$ coming from the two projections $\pi_i: \mathbb{F}_0 \to \mathbb{P}^1$ for $i=1,2$. The fibration $\pi_1$ is the alternate fibration from Section~\ref{ssec:alt}. The second elliptic fibration arises from the projection $\pi_2$ and has the reducible fibers of type $\widetilde{E}_8$ and $\widetilde{D}_6$.
\par There are exactly two ways of embedding the disjoint reducible fibers of type $\widetilde{E}_8$ and $\widetilde{D}_6$ into the diagram~(\ref{exdiagg77}). These two ways are depicted in Figure~\ref{fig:e8d6fib}. In the case of Figure~\ref{Fig3a}, we have
\beq
\begin{split}
{\color{green}\widetilde{E}_8}=   \langle a_1,  a_2,    a_3,   L_2,   a_4,  a_5  ,  a_6 ,  a_7 ; a_8 \rangle \,, \quad
{\color{blue}\widetilde{D}_6} =   \langle R_1, b_5,  b_4,  b_3, b_2, b_1 ; L_1 \rangle \, .
\end{split}
\eeq
Thus, the smooth fiber class is given by
\beq
 \mathrm{F}_{\text{bfd}}^{(a)} =  L_1 + b_1 + 2 b_2 + 2 b_3 + 2 b_4 + b_5 +  R_1  \,,
\eeq
and the class of a section is ${\color{red} a_9}$. 
\begin{figure}
\hspace*{-1.5cm}
\begin{subfigure}{.5\textwidth}  
\centering
\begin{equation*}
\scalemath{\MyScaleSmall}{
\xymatrix @-0.9pc  {
& & \stackrel{L_3}{\bullet} \ar @{=}[rrrrrrrrrr]& &  & & &    & & & & & \color{blue}{\stackrel{R_1}{\bullet}} &  &  \\ 
 & &  & &  & & & &   &    & &  & & \\
&  
 \color{green}{\stackrel{a_1}{\bullet}} \ar @{-} [r] \ar @{-}[ddr] \ar @{-}[uur] &
 \color{green}{\stackrel{a_2}{\bullet}} \ar @{-} [r]  &
 \color{green}{\stackrel{a_3}{\bullet}} \ar @{-} [r]  \ar @{-} [d] &
 \color{green}{\stackrel{a_4}{\bullet}} \ar @{-} [r] &
 \color{green}{\stackrel{a_5}{\bullet}} \ar @{-} [r] &
 \color{green}{\stackrel{a_6}{\bullet}} \ar @{-} [r] &
 \color{green}{\stackrel{a_7}{\bullet}} \ar @{-} [r] &
 \color{green}{\stackrel{a_8}{\bullet}} \ar @{-} [r] &
 \color{red}{\stackrel{a_9}{\bullet}} \ar @{-} [r] &
 \color{blue}{\stackrel{L_1}{\bullet}} \ar @{-} [r] &
 \color{blue}{\stackrel{b_2}{\bullet}} \ar @{-} [r] \ar @{-} [d] &
 \color{blue}{\stackrel{b_3}{\bullet}} & \ \color{blue}{\stackrel{b_4}{\bullet}} \ar @{-} [l] \ar @{-}[ddl] \ar @{-}[uul]
 & \\
 & &  &  \color{green}{\stackrel{L_2}{\bullet}} &  & & & &   &    & &  \color{blue}{\stackrel{b_1}{\bullet}}  & & \\
 & & \stackrel{R_2}{\bullet} \ar @{=}[rrrrrrrrrr]& &  & & &    & & & & &    \color{blue}{\stackrel{b_5}{\bullet}} & & \\
 } }
\end{equation*}
\caption{\phantom{A}}
\label{Fig3a}
\end{subfigure}%
\quad
\begin{subfigure}{.5\textwidth}
\centering%
\begin{equation*}
\scalemath{\MyScaleSmall}{
\xymatrix @-0.9pc  {
& & \color{blue}{\stackrel{L_3}{\bullet}} \ar @{=}[rrrrrrrrrr]& &  & & &    & & & & & \stackrel{R_1}{\bullet} &  &  \\ 
 & &  & &  & & & &   &    & &  & & \\
&  
 \color{blue}{\stackrel{a_1}{\bullet}} \ar @{-} [r] \ar @{-}[ddr] \ar @{-}[uur] &
 \color{blue}{\stackrel{a_2}{\bullet}} \ar @{-} [r]  &
 \color{blue}{\stackrel{a_3}{\bullet}} \ar @{-} [r]  \ar @{-} [d] &
 \color{blue}{\stackrel{a_4}{\bullet}} \ar @{-} [r] &
 \color{red}{\stackrel{a_5}{\bullet}} \ar @{-} [r] &
 \color{green}{\stackrel{a_6}{\bullet}} \ar @{-} [r] &
 \color{green}{\stackrel{a_7}{\bullet}} \ar @{-} [r] &
 \color{green}{\stackrel{a_8}{\bullet}} \ar @{-} [r] &
 \color{green}{\stackrel{a_9}{\bullet}} \ar @{-} [r] &
 \color{green}{\stackrel{L_1}{\bullet}} \ar @{-} [r] &
 \color{green}{\stackrel{b_2}{\bullet}} \ar @{-} [r] \ar @{-} [d] &
 \color{green}{\stackrel{b_3}{\bullet}} & \ \color{green}{\stackrel{b_4}{\bullet}} \ar @{-} [l] \ar @{-}[ddl] \ar @{-}[uul]
 & \\
 & &  &  \color{blue}{\stackrel{L_2}{\bullet}} &  & & & &   &    & &  \color{green}{\stackrel{b_1}{\bullet}}  & & \\
 & & \color{blue}{\stackrel{R_2}{\bullet}} \ar @{=}[rrrrrrrrrr]& &  & & &    & & & & &    \stackrel{b_5}{\bullet} & & \\
 } }
\end{equation*}
\caption{\phantom{B}}
\label{Fig3b}
\end{subfigure}%
\caption{The \emph{base-fiber dual fibrations} with fibers of type $\widetilde{E}_8$ and $\widetilde{D}_6$}
\label{fig:e8d6fib}
\end{figure}%
Making the substitutions 
\beq
\begin{split}
\mathbf{X}= 3 uv (x+6\gamma \varepsilon uv^3z)\,, \quad  \mathbf{Y}=  y \,, \\ 
\mathbf{Z} = 6 v^2 (\varepsilon x-6\gamma \varepsilon^2 uv^3z-18 \zeta u^2 v^2 z)\,, \quad \mathbf{W} = 108 u^3 v^3 z \,,
\end{split}
\eeq
into Equation~(\ref{mainquartic}), is compatible with $L_3(u, v)=0$. Here, $L_3(u, v)=u \mathbf{Z} - v (2 \varepsilon \mathbf{X} - \zeta \mathbf{W})=0$ for $[u:v] \in \mathbb{P}^1$ is the pencil of planes containing the line $L_3$. We obtain a Jacobian elliptic fibration  $\pi_{\mathrm{bfd}}: \mathcal{X} \to \mathbb{P}^1$ with fiber $\mathcal{X}_{[u:v]}$, given by
the equation
\beq
\label{eqn:bfd}
\mathcal{X}_{[u:v]}: \quad y^2 z = x^3 + F(u, v) \, x z^2 + G(u, v) \, z^3 \,,
\eeq
admitting the section $\sigma_{\text{bfd}} : [x:y:z] = [0:1:0]$, and with the discriminant 
\beq
\Delta(u, v) =  4  F^3 + 27  G^2 =  - 2^6 3^{12} u^8 v^{10}  P(u, v) \,,
\eeq
where
\beq
\begin{split}
F(u, v)  = & \; - 108\, u^2 v^4  \Big(9 \alpha  u^2 - 3 (\gamma \zeta+\delta \varepsilon)uv + \gamma^2 \varepsilon^2 v^2\Big) \,, \\
G(u, v) = & \; - 216 u^3 v^5 \Big( 27 u^4 + 54 \beta u^3 v + 27(\alpha \gamma \varepsilon+ \delta \zeta) u^2 v^2 \\
& \; - 9 \gamma \varepsilon (\gamma\zeta+\delta\varepsilon) u v^3 + 2 \gamma^3 \epsilon^3 v^4 \Big) \,,
\end{split}
\eeq
and $P(u, v) =   \gamma^2\varepsilon^2 (\gamma\zeta-\delta\varepsilon)^2 v^6 + O(u)$ is an irreducible homogeneous polynomial of degree six.
We have the following:
\begin{lemma}
\label{lem:fib_bfd}
Equation~(\ref{eqn:bfd}) defines a Jacobian elliptic fibration with the singular fibers $6 I_1 + I_2^* + II^*$ and a Mordell-Weil group $\operatorname{MW}(\pi_{\text{bfd}}, \sigma_{\text{bfd}}) =  \{\mathbb{I}\}$.
\end{lemma}
\begin{proof}
The proof easily follows by checking the Kodaira type of the singular fibers at $P(u, v)=0$, $u=0$, and $v=0$. This is done by checking the vanishing degrees of $(F, G, \Delta)$ at these points. These vanishing degrees are $(0,0,1)$ for an $I_1$-fiber, $(2,3,8)$ for an $I^*_2$-fiber, and $(4,5,10)$ for an $II^*$-fiber.
\end{proof}
\par Applying the Nikulin involution in Proposition~\ref{NikulinInvolution}, we obtain the fiber configuration in Figure~\ref{Fig3b} with
\beq
\begin{split}
{\color{blue}\widetilde{D}_6} =   \langle L_3,  R_2, a_1,  a_2,  a_3, L_2; a_4 \rangle \, , \quad
{\color{green}\widetilde{E}_8}=   \langle b_4,  b_3,    b_1,  b_2,  L_1,   a_9,  a_8  ,  a_7 ; a_6 \rangle \,.
\end{split}
\eeq
The smooth fiber class is given by
\beq
 \mathrm{F}^{(b)}_{\text{bfd}} =  R_2 +  L_2 +  L_3 + 2 a_1 + 2 a_2 + 2 a_3 +  a_4  \,,
 \eeq
and the class of the section is ${\color{red} a_5}$. Using the polarizing divisor $\mathcal{H}$ in Equation~(\ref{linepolariz}), one checks that
\beq
\begin{split}
& 2 \mathcal{H} -  \mathrm{F}^{(b)}_{\text{bfd}} - L_1  - 2L_2  \\
\equiv \; & 2 a_1 + 4 a_2 + 6 a_3 + 6 a_4 + 6  a_5 + 5 a_6 + 4 a_7 + 3 a_8 + 2 a_9 + b_1 + 2 b_2 + \dots + 2 b_5\,,
\end{split} 
\eeq
which shows that the base-fiber dual fibration is also induced by intersecting the quartic with the pencil of quadratic surfaces, denoted by $C_2(u, v)=0$ with $[u:v] \in \mathbb{P}^1$, containing the lines $L_1, L_2$ and also being tangent to $L_2$. A computation yields
\beq
 C_2(u, v) = v \mathbf{Z} \big( 2 \gamma \mathbf{X}- \delta \mathbf{W} \big)- u \mathbf{W}^2 =0 \,.
\eeq
\subsection{The maximal fibration}
\label{ssec:max}
\par There are exactly two ways of embedding one reducible ADE-type fiber of the biggest possible rank, namely a fiber of type $\widetilde{D}_{14}$, into the diagram~(\ref{exdiagg77}). These two ways are depicted in Figure~\ref{fig:d14fib}. In the case of Figure~\ref{Fig4a}, we have
\beq
\begin{split}
{\color{brown}\widetilde{D}_{14}} =   \langle R_2, L_3 ,  a_1,  \dots , a_9, L_1, b_2, b_1 ;  b_3 \rangle \, .
\end{split}
\eeq
Thus, the smooth fiber class is given by
\beq
 \mathrm{F}_{\text{max}}^{(a)} =  R_2 + L_3 + 2 a_1 + \dots+ 2 a_9 + 2  L_1 + 2 b_2 + b_1 + b_3 \,,
\eeq
and the class of a section is ${\color{red} b_4}$. Using the polarizing divisor $\mathcal{H}$ in Equation~(\ref{linepolariz}), one checks that
\beq
 2 \mathcal{H} -  \mathrm{F}_{\text{max}}^{(a)} - R_1 \equiv b_1 + 2 b_2 + 3 b_3 + 4 b_4 + 3 b_5    \,.
\eeq
This shows that the elliptic fibration with section is induced by intersecting the quartic surface $\mathcal{Q}(\alpha, \beta, \gamma, \delta , \varepsilon, \zeta)$ with a special pencil of quadric surfaces containing the curve $R_1$, denoted by $C_3(u, v)=0$ with $[u:v] \in \mathbb{P}^1$. Because of the identity
\beq
 3 \mathcal{H} -  \mathrm{F}_{\text{max}}^{(a)} - R_1 - R_2 - L_1 \equiv a_1 + \dots + a_9 +   2 b_1 + 4 b_2 + 5 b_3 + 6 b_4 + 5 b_5  \,,
\eeq
the quadric surfaces must also define a pencil of reducible cubic surfaces containing the curves $L_1$, $R_1$, and $R_2$. This pencil turns out to be $(2 \gamma \mathbf{X} - \delta \mathbf{W}) C_3(u, v)=0$ with
\beq 
\label{eq:conic}
\begin{split}
C_3(u, v)= v \Big( 2 \gamma^2 \delta \varepsilon \zeta \mathbf{X} \mathbf{Z} + (6\alpha\gamma\delta\varepsilon\zeta +4 \beta \gamma \delta \varepsilon^2  + 4 \beta \gamma^2 \varepsilon \zeta+2 \delta^2 \zeta^2) \mathbf{X} \mathbf{W} - \gamma\delta^2\varepsilon \zeta \mathbf{Z} \mathbf{W} \\
+ 2 \gamma \delta \varepsilon \zeta \mathbf{Y}^2  - (8 \beta \gamma^2 \varepsilon^2 + 4 \delta^2 \varepsilon \zeta + 4 \gamma \delta \zeta^2) \mathbf{X}^2 \Big) 
+ u (2 \gamma \mathbf{X} - \delta \mathbf{W})(2 \varepsilon \mathbf{X} - \zeta \mathbf{W}) =0 \,.
\end{split}
\eeq
\begin{figure}
\hspace*{-1.5cm}
\begin{subfigure}{.5\textwidth}  
\centering
\begin{equation*}
\scalemath{\MyScaleSmall}{
\xymatrix @-0.9pc  {
& & \color{brown}{\stackrel{L_3}{\bullet}} \ar @{=}[rrrrrrrrrr]& &  & & &    & & & & & \stackrel{R_1}{\bullet} &  &  \\ 
 & &  & &  & & & &   &    & &  & & \\
&  
 \color{brown}{\stackrel{a_1}{\bullet}} \ar @{-} [r] \ar @{-}[ddr] \ar @{-}[uur] &
 \color{brown}{\stackrel{a_2}{\bullet}} \ar @{-} [r]  &
 \color{brown}{\stackrel{a_3}{\bullet}} \ar @{-} [r]  \ar @{-} [d] &
 \color{brown}{\stackrel{a_4}{\bullet}} \ar @{-} [r] &
 \color{brown}{\stackrel{a_5}{\bullet}} \ar @{-} [r] &
 \color{brown}{\stackrel{a_6}{\bullet}} \ar @{-} [r] &
 \color{brown}{\stackrel{a_7}{\bullet}} \ar @{-} [r] &
 \color{brown}{\stackrel{a_8}{\bullet}} \ar @{-} [r] &
 \color{brown}{\stackrel{a_9}{\bullet}} \ar @{-} [r] &
 \color{brown}{\stackrel{L_1}{\bullet}} \ar @{-} [r] &
 \color{brown}{\stackrel{b_2}{\bullet}} \ar @{-} [r] \ar @{-} [d] &
 \color{brown}{\stackrel{b_3}{\bullet}} & \ \color{red}{\stackrel{b_4}{\bullet}} \ar @{-} [l] \ar @{-}[ddl] \ar @{-}[uul]
 & \\
 & &  &  \stackrel{L_2}{\bullet} &  & & & &   &    & &  \color{brown}{\stackrel{b_1}{\bullet}}  & & \\
 & &  \color{brown}{\stackrel{R_2}{\bullet}} \ar @{=}[rrrrrrrrrr]& &  & & &    & & & & &    \stackrel{b_5}{\bullet} & & \\
 } }
\end{equation*}
\caption{\phantom{A}}
\label{Fig4a}
\end{subfigure}%
\quad
\begin{subfigure}{.5\textwidth}
\centering%
\begin{equation*}
\scalemath{\MyScaleSmall}{
\xymatrix @-0.9pc  {
& & \stackrel{L_3}{\bullet} \ar @{=}[rrrrrrrrrr]& &  & & &    & & & & & \color{brown}{\stackrel{R_1}{\bullet}} &  &  \\ 
 & &  & &  & & & &   &    & &  & & \\
&  
 \color{red}{\stackrel{a_1}{\bullet}} \ar @{-} [r] \ar @{-}[ddr] \ar @{-}[uur] &
 \color{brown}{\stackrel{a_2}{\bullet}} \ar @{-} [r]  &
 \color{brown}{\stackrel{a_3}{\bullet}} \ar @{-} [r]  \ar @{-} [d] &
 \color{brown}{\stackrel{a_4}{\bullet}} \ar @{-} [r] &
 \color{brown}{\stackrel{a_5}{\bullet}} \ar @{-} [r] &
 \color{brown}{\stackrel{a_6}{\bullet}} \ar @{-} [r] &
 \color{brown}{\stackrel{a_7}{\bullet}} \ar @{-} [r] &
 \color{brown}{\stackrel{a_8}{\bullet}} \ar @{-} [r] &
 \color{brown}{\stackrel{a_9}{\bullet}} \ar @{-} [r] &
 \color{brown}{\stackrel{L_1}{\bullet}} \ar @{-} [r] &
 \color{brown}{\stackrel{b_2}{\bullet}} \ar @{-} [r] \ar @{-} [d] &
 \color{brown}{\stackrel{b_3}{\bullet}} & \ \color{brown}{\stackrel{b_4}{\bullet}} \ar @{-} [l] \ar @{-}[ddl] \ar @{-}[uul]
 & \\
 & &  &  \color{brown}{\stackrel{L_2}{\bullet}} &  & & & &   &    & &  \stackrel{b_1}{\bullet}  & & \\
 & & \stackrel{R_2}{\bullet} \ar @{=}[rrrrrrrrrr]& &  & & &    & & & & &    \color{brown}{\stackrel{b_5}{\bullet}} & & \\
 } }
\end{equation*}
\caption{\phantom{B}}
\label{Fig4b}
\end{subfigure}%
\caption{The \emph{maximal fibrations} with a fiber of type $\widetilde{D}_{14}$}
\label{fig:d14fib}
\end{figure}%
\par Making the substitutions 
\beq
\begin{split}
\mathbf{X}= \delta\zeta v \big((2\beta \gamma \varepsilon v-u)x -2 \gamma \delta^5 \varepsilon \zeta^5 v^5 z\big)\,, \quad  \mathbf{Y}=  y \,, \quad \mathbf{W} =  2 \delta^2\zeta^2 v^2 x \,,
\end{split}
\eeq
and $\mathbf{Z}=\mathbf{Z}(x, y, z, u, v)$, obtained by solving Equation~(\ref{eq:conic}) for $\mathbf{Z}$, determines a Jacobian elliptic fibration  $\pi_{\mathrm{max}}: \mathcal{X} \to \mathbb{P}^1$ with fiber $\mathcal{X}_{[u:v]}$, given by the equation
\beq
\label{eqn:max}
\mathcal{X}_{[u:v]}: \quad y^2 z = x^3 + a(u, v) \, x^2 z + b(u, v) \, x z^2 + c(u, v) \, z^3 \,,
\eeq
admitting the section $\sigma_{\text{max}} : [x:y:z] = [0:1:0]$, and with the discriminant 
\beq
\begin{split}
\Delta(u, v) & =  b^2 \big(a^2 - 4 b\big) - 2 a c  \big(2 a^2-9 b \big) -27 c^2  = 64 \delta^{16} \zeta^{16} v^{16} d(u, v)\,,
\end{split}
\eeq
where
\beq
\begin{split}
a(u, v)  & = - 2 \delta \zeta v \Big( u^3 - 6 \beta \gamma \varepsilon u^2 v + 3(4 \beta^2\gamma^2 \varepsilon^2- \alpha\delta^2\zeta^2) uv^2\\
& \quad - 2 \beta ( 4 \beta^2\gamma^3 \varepsilon^3 - 3 \alpha\gamma \delta^2 \varepsilon\zeta^2-\delta^3\zeta^3) v^3 \Big)\,, \\
b(u, v)  & = - 4 \delta^6 \zeta^6 v^6 \Big( 2 \gamma \varepsilon u^2 - (8 \beta \gamma^2 \varepsilon^2+ \gamma\delta\zeta^2+\delta^2\varepsilon\zeta) uv \\
& \quad +(8 \beta^2 \gamma^3 \varepsilon^3-3 \alpha\gamma\delta^2\varepsilon\zeta^2 +2  \beta \gamma^2 \delta \varepsilon \zeta^2 + 2 \beta\gamma \delta^2\varepsilon^2\zeta -\delta^3\zeta^3)v^2 \Big) \,, \\
c(u, v) &= - 8 \gamma \delta^{11} \varepsilon \zeta^{11} v^{11} \Big( \gamma \varepsilon u - (2 \beta \gamma^2 \varepsilon^2 + \gamma\delta \zeta^2+\delta^2\varepsilon\zeta)v\Big) \,,
\end{split}
\eeq
and $d(u, v) =  (\gamma\zeta - \delta\varepsilon)^2 u^8 + O(v)$ is an irreducible homogeneous polynomial of degree eight.
We have the following:
\begin{lemma}
\label{lem:fib_max}
Equation~(\ref{eqn:max}) defines a Jacobian elliptic fibration with the singular fibers $8 I_1 + I_{10}^*$ and a Mordell-Weil group $\operatorname{MW}(\pi_{\text{max}}, \sigma_{\text{max}}) =  \{\mathbb{I}\}$.
\end{lemma}
\begin{proof}
The proof easily follows by checking the Kodaira type of the singular fibers at $d(u, v)=0$ and $v=0$.
\end{proof}
\par Applying the Nikulin involution in Proposition~\ref{NikulinInvolution}, we obtain the fiber configuration in Figure~\ref{Fig4b} with
\beq
\begin{split}
{\color{brown}\widetilde{D}_{14}} =   \langle b_5, R_1, b_4, b_3, b_ 2, L_1 ,  a_9,  \dots , a_3, L_2 ;  a_2 \rangle \, .
\end{split}
\eeq
The smooth fiber class is given by
\beq
 \mathrm{F}_{\text{max}}^{(b)} =  R_1 + L_2 + 2  L_1 + a_2 + 2a_3 + \dots + 2a_9 + 2 b_2 + 2 b_3 + 2 b_4 + b_5 \,,
 \eeq
and the class of the section is ${\color{red} a_1}$. Using the polarizing divisor $\mathcal{H}$ in Equation~(\ref{linepolariz}), one checks that the elliptic fibration is also induced by intersecting the quartic surface $\mathcal{Q}(\alpha, \beta, \gamma, \delta , \varepsilon, \zeta)$ with a special pencil of cubic surfaces containing the curves $L_2$, $L_3$, $R_2$, denoted by $T(u, v)=0$ with $[u:v] \in \mathbb{P}^1$. This pencil of cubic surfaces is
\beq
 T(u, v) = C_3(u, v)\mathbf{Z}  - \gamma \delta\varepsilon \zeta \Big( C_2(u, v)  \mathbf{Z} - L_3(u, v) \mathbf{W}^2 \Big) \,.
\eeq
\bibliographystyle{amsplain}
\bibliography{references}{}

\begin{bibdiv}
\begin{biblist}

\bib{MR3882710}{incollection}{
      author={Balestrieri, Francesca},
      author={Desjardins, Julie},
      author={Garbagnati, Alice},
      author={Maistret, C\'{e}line},
      author={Salgado, Cec\'{\i}lia},
      author={Vogt, Isabel},
       title={Elliptic fibrations on covers of the elliptic modular surface of
  level 5},
        date={2018},
   booktitle={Women in numbers {E}urope {II}},
      series={Assoc. Women Math. Ser.},
      volume={11},
   publisher={Springer, Cham},
       pages={159\ndash 197},
  url={https://mathscinet-ams-org.ezproxy.lib.uconn.edu/mathscinet-getitem?mr=3882710},
      review={\MR{3882710}},
}

\bib{MR4015343}{article}{
      author={Clingher, A.},
      author={Malmendier, A.},
      author={Shaska, T.},
       title={Six line configurations and string dualities},
        date={2019},
        ISSN={0010-3616},
     journal={Comm. Math. Phys.},
      volume={371},
      number={1},
       pages={159\ndash 196},
         url={https://mathscinet.ams.org/mathscinet-getitem?mr=4015343},
      review={\MR{4015343}},
}

\bib{Clingher:2021}{article}{
      author={Clingher, Adiran},
      author={Malmendier, Andreas},
       title={{On Picard lattices of Jacobian elliptic K3 surfaces}},
      eprint={arXiv:2109.01929},
}

\bib{MR2369941}{article}{
      author={Clingher, Adrian},
      author={Doran, Charles~F.},
       title={Modular invariants for lattice polarized {$K3$} surfaces},
        date={2007},
        ISSN={0026-2285},
     journal={Michigan Math. J.},
      volume={55},
      number={2},
       pages={355\ndash 393},
         url={http://dx.doi.org/10.1307/mmj/1187646999},
      review={\MR{2369941 (2009a:14049)}},
}

\bib{MR2355598}{article}{
      author={Clingher, Adrian},
      author={Doran, Charles~F.},
       title={On {$K3$} surfaces with large complex structure},
        date={2007},
        ISSN={0001-8708},
     journal={Adv. Math.},
      volume={215},
      number={2},
       pages={504\ndash 539},
         url={https://mathscinet.ams.org/mathscinet-getitem?mr=2355598},
      review={\MR{2355598}},
}

\bib{MR2824841}{article}{
      author={Clingher, Adrian},
      author={Doran, Charles~F.},
       title={Note on a geometric isogeny of {K}3 surfaces},
        date={2011},
        ISSN={1073-7928},
     journal={Int. Math. Res. Not. IMRN},
      number={16},
       pages={3657\ndash 3687},
      review={\MR{2824841 (2012f:14072)}},
}

\bib{MR3201823}{article}{
      author={Comparin, Paola},
      author={Garbagnati, Alice},
       title={Van {G}eemen-{S}arti involutions and elliptic fibrations on
  {$K3$} surfaces double cover of {$\mathbb{P}^2$}},
        date={2014},
        ISSN={0025-5645},
     journal={J. Math. Soc. Japan},
      volume={66},
      number={2},
       pages={479\ndash 522},
  url={https://mathscinet-ams-org.ezproxy.lib.uconn.edu/mathscinet-getitem?mr=3201823},
      review={\MR{3201823}},
}

\bib{MR728992}{incollection}{
      author={Dolgachev, Igor},
       title={Integral quadratic forms: applications to algebraic geometry
  (after {V}. {N}ikulin)},
        date={1983},
   booktitle={Bourbaki seminar, {V}ol. 1982/83},
      series={Ast\'{e}risque},
      volume={105},
   publisher={Soc. Math. France, Paris},
       pages={251\ndash 278},
         url={https://mathscinet.ams.org/mathscinet-getitem?mr=728992},
      review={\MR{728992}},
}

\bib{MR1420220}{article}{
      author={Dolgachev, Igor~V.},
       title={Mirror symmetry for lattice polarized {$K3$} surfaces},
        date={1996},
        ISSN={1072-3374},
     journal={J. Math. Sci.},
      volume={81},
      number={3},
       pages={2599\ndash 2630},
         url={http://dx.doi.org/10.1007/BF02362332},
        note={Algebraic geometry, 4},
      review={\MR{1420220 (97i:14024)}},
}

\bib{FestiVeniani20}{article}{
      author={Festi, Dino},
      author={Veniani, Davide~Cesare},
       title={Counting ellptic fibrations on {K}3 surfaces,},
     journal={2020, arXiv:2102.09411},
}

\bib{MR4130832}{article}{
      author={Garbagnati, Alice},
      author={Salgado, Cec\'{\i}lia},
       title={Elliptic fibrations on {K}3 surfaces with a non-symplectic
  involution fixing rational curves and a curve of positive genus},
        date={2020},
        ISSN={0213-2230},
     journal={Rev. Mat. Iberoam.},
      volume={36},
      number={4},
       pages={1167\ndash 1206},
  url={https://mathscinet-ams-org.ezproxy.lib.uconn.edu/mathscinet-getitem?mr=4130832},
      review={\MR{4130832}},
}

\bib{MR578868}{inproceedings}{
      author={Inose, Hiroshi},
       title={Defining equations of singular {$K3$} surfaces and a notion of
  isogeny},
        date={1978},
   booktitle={Proceedings of the {I}nternational {S}ymposium on {A}lgebraic
  {G}eometry ({K}yoto {U}niv., {K}yoto, 1977)},
   publisher={Kinokuniya Book Store, Tokyo},
       pages={495\ndash 502},
         url={https://mathscinet.ams.org/mathscinet-getitem?mr=578868},
      review={\MR{578868}},
}

\bib{MR2254405}{article}{
      author={Kloosterman, Remke},
       title={Classification of all {J}acobian elliptic fibrations on certain
  {$K3$} surfaces},
        date={2006},
        ISSN={0025-5645},
     journal={J. Math. Soc. Japan},
      volume={58},
      number={3},
       pages={665\ndash 680},
         url={https://mathscinet.ams.org/mathscinet-getitem?mr=2254405},
      review={\MR{2254405}},
}

\bib{MR1029967}{article}{
      author={Kondo, Shigeyuki},
       title={Algebraic {$K3$} surfaces with finite automorphism groups},
        date={1989},
        ISSN={0027-7630},
     journal={Nagoya Math. J.},
      volume={116},
       pages={1\ndash 15},
         url={https://mathscinet.ams.org/mathscinet-getitem?mr=1029967},
      review={\MR{1029967}},
}

\bib{MR1139659}{article}{
      author={Kondo, Shigeyuki},
       title={Automorphisms of algebraic {$K3$} surfaces which act trivially on
  {P}icard groups},
        date={1992},
        ISSN={0025-5645},
     journal={J. Math. Soc. Japan},
      volume={44},
      number={1},
       pages={75\ndash 98},
         url={https://mathscinet.ams.org/mathscinet-getitem?mr=1139659},
      review={\MR{1139659}},
}

\bib{MR1204828}{article}{
      author={Matsumoto, Keiji},
       title={Theta functions on the bounded symmetric domain of type
  {$I_{2,2}$} and the period map of a {$4$}-parameter family of {$K3$}
  surfaces},
        date={1993},
        ISSN={0025-5831},
     journal={Math. Ann.},
      volume={295},
      number={3},
       pages={383\ndash 409},
         url={http://dx.doi.org/10.1007/BF01444893},
      review={\MR{1204828 (94m:32049)}},
}

\bib{MR0357410}{article}{
      author={Nikulin, V.~V.},
       title={An analogue of the {T}orelli theorem for {K}ummer surfaces of
  {J}acobians},
        date={1974},
        ISSN={0373-2436},
     journal={Izv. Akad. Nauk SSSR Ser. Mat.},
      volume={38},
       pages={22\ndash 41},
         url={https://mathscinet.ams.org/mathscinet-getitem?mr=0357410},
      review={\MR{0357410}},
}

\bib{MR0429917}{article}{
      author={Nikulin, V.~V.},
       title={Kummer surfaces},
        date={1975},
        ISSN={0373-2436},
     journal={Izv. Akad. Nauk SSSR Ser. Mat.},
      volume={39},
      number={2},
       pages={278\ndash 293, 471},
         url={https://mathscinet.ams.org/mathscinet-getitem?mr=0429917},
      review={\MR{0429917}},
}

\bib{MR544937}{article}{
      author={Nikulin, V.~V.},
       title={Finite groups of automorphisms of {K}\"{a}hlerian {$K3$}
  surfaces},
        date={1979},
        ISSN={0134-8663},
     journal={Trudy Moskov. Mat. Obshch.},
      volume={38},
       pages={75\ndash 137},
         url={https://mathscinet.ams.org/mathscinet-getitem?mr=544937},
      review={\MR{544937}},
}

\bib{MR525944}{article}{
      author={Nikulin, V.~V.},
       title={Integer symmetric bilinear forms and some of their geometric
  applications},
        date={1979},
        ISSN={0373-2436},
     journal={Izv. Akad. Nauk SSSR Ser. Mat.},
      volume={43},
      number={1},
       pages={111\ndash 177, 238},
         url={https://mathscinet.ams.org/mathscinet-getitem?mr=525944},
      review={\MR{525944}},
}

\bib{MR556762}{article}{
      author={Nikulin, V.~V.},
       title={Quotient-groups of groups of automorphisms of hyperbolic forms of
  subgroups generated by {$2$}-reflections},
        date={1979},
        ISSN={0002-3264},
     journal={Dokl. Akad. Nauk SSSR},
      volume={248},
      number={6},
       pages={1307\ndash 1309},
         url={https://mathscinet.ams.org/mathscinet-getitem?mr=556762},
      review={\MR{556762}},
}

\bib{MR633160}{incollection}{
      author={Nikulin, V.~V.},
       title={Quotient-groups of groups of automorphisms of hyperbolic forms by
  subgroups generated by {$2$}-reflections. {A}lgebro-geometric applications},
        date={1981},
   booktitle={Current problems in mathematics, {V}ol. 18},
   publisher={Akad. Nauk SSSR, Vsesoyuz. Inst. Nauchn. i Tekhn. Informatsii,
  Moscow},
       pages={3\ndash 114},
         url={https://mathscinet.ams.org/mathscinet-getitem?mr=633160},
      review={\MR{633160}},
}

\bib{MR633160b}{article}{
      author={Nikulin, V.~V.},
       title={Factor groups of groups of automorphisms of hyperbolic forms with
  respect to subgroups generated by 2-reflections. {A}lgebrogeometric
  applications},
        date={1983},
     journal={Journal of Soviet Mathematics},
      volume={22},
      number={4},
       pages={1401\ndash 1475},
         url={https://doi.org/10.1007/BF01094757},
}

\bib{MR752938}{incollection}{
      author={Nikulin, V.~V.},
       title={{$K3$} surfaces with a finite group of automorphisms and a
  {P}icard group of rank three},
        date={1984},
      volume={165},
       pages={119\ndash 142},
         url={https://mathscinet.ams.org/mathscinet-getitem?mr=752938},
        note={Algebraic geometry and its applications},
      review={\MR{752938}},
}

\bib{MR1813537}{article}{
      author={Shimada, Ichiro},
       title={On elliptic {$K3$} surfaces},
        date={2000},
        ISSN={0026-2285},
     journal={Michigan Math. J.},
      volume={47},
      number={3},
       pages={423\ndash 446},
         url={https://mathscinet.ams.org/mathscinet-getitem?mr=1813537},
      review={\MR{1813537}},
}

\bib{MR1081832}{article}{
      author={Shioda, Tetsuji},
       title={On the {M}ordell-{W}eil lattices},
        date={1990},
        ISSN={0010-258X},
     journal={Comment. Math. Univ. St. Paul.},
      volume={39},
      number={2},
       pages={211\ndash 240},
         url={https://mathscinet.ams.org/mathscinet-getitem?mr=1081832},
      review={\MR{1081832}},
}

\bib{MR2279280}{article}{
      author={Shioda, Tetsuji},
       title={Kummer sandwich theorem of certain elliptic {$K3$} surfaces},
        date={2006},
        ISSN={0386-2194},
     journal={Proc. Japan Acad. Ser. A Math. Sci.},
      volume={82},
      number={8},
       pages={137\ndash 140},
         url={https://mathscinet.ams.org/mathscinet-getitem?mr=2279280},
      review={\MR{2279280}},
}

\bib{MR2274533}{article}{
      author={van Geemen, Bert},
      author={Sarti, Alessandra},
       title={Nikulin involutions on {$K3$} surfaces},
        date={2007},
        ISSN={0025-5874},
     journal={Math. Z.},
      volume={255},
      number={4},
       pages={731\ndash 753},
         url={https://mathscinet.ams.org/mathscinet-getitem?mr=2274533},
      review={\MR{2274533}},
}

\bib{MR2682724}{article}{
      author={Vinberg, E.~B.},
       title={On automorphic forms on symmetric domains of type {IV}},
        date={2010},
        ISSN={0042-1316},
     journal={Uspekhi Mat. Nauk},
      volume={65},
      number={3(393)},
       pages={193\ndash 194},
         url={https://mathscinet.ams.org/mathscinet-getitem?mr=2682724},
      review={\MR{2682724}},
}

\bib{MR2718942}{article}{
      author={Vinberg, E.~B.},
       title={Some free algebras of automorphic forms on symmetric domains of
  type {IV}},
        date={2010},
        ISSN={1083-4362},
     journal={Transform. Groups},
      volume={15},
      number={3},
       pages={701\ndash 741},
         url={https://mathscinet.ams.org/mathscinet-getitem?mr=2718942},
      review={\MR{2718942}},
}

\bib{MR3235787}{article}{
      author={Vinberg, E.~B.},
       title={On the algebra of {S}iegel modular forms of genus 2},
        date={2013},
        ISSN={0077-1554},
     journal={Trans. Moscow Math. Soc.},
       pages={1\ndash 13},
         url={https://mathscinet.ams.org/mathscinet-getitem?mr=3235787},
      review={\MR{3235787}},
}

\end{biblist}
\end{bibdiv}
\end{document}